\documentclass{article}

\usepackage[english]{babel}

\usepackage[letterpaper,top=2cm,bottom=2cm,left=3cm,right=3cm,marginparwidth=1.75cm]{geometry}

\usepackage{amsmath}
\usepackage{amssymb}
\usepackage{amsthm}
\usepackage{graphicx}
\usepackage{enumitem}
\usepackage{tikz}
\usepackage[colorlinks=true, allcolors=blue]{hyperref}

\usepackage{xcolor}
\usepackage[all]{xy}

\newcommand\sbmattrix[4]{\textnormal{\scriptsize$\left(\begin{array}{cc}#1&#2\\#3&#4\end{array}\right)$\normalsize}}

\newtheorem{theo}{Theorem}
\newtheorem{prop}{Proposition}[section]
\newtheorem{lem}[prop]{Lemma}
\newtheorem{cor}{Corollary}[prop]
\theoremstyle{definition}
\newtheorem{ex}[prop]{Example}
\newtheorem{rem}[prop]{Remark}

\title{On some representations of Metacyclic 
groups whose integral forms can be computed 
from a single residual representation}
\author{
  Aguil\'o-Vidal, Bruno\\
  \texttt{bruno.aguilo@udp.cl}
  \and
    Arenas-Carmona, Luis\\
  \texttt{learenas@u.uchile.cl}
  \and
  Saavedra-Lagos, Mat\'ias\\
  \texttt{matias.saavedra.l@ug.uchile.cl}
}

\begin{document}
\maketitle

\begin{abstract}
 In a previous work we computed the number of integral forms, over its field of 
definition, of an irreducible representation of a dihedral group. 
Here we apply the theory of Bruhat-Tits buildings to 
give similar formulas for a wider family of metacyclic groups that have 
representations of arbitrarily large dimension. Occasionally, these formulas 
can be extended to groups containing a subgroup in that family. 
\end{abstract}

MSC-class: 11R33, 20C10 (Primary)  11R56 (Secondary).


\section{Introduction}

 The study of representations over fields of characteristic zero takes
 advantage of the highly simplified structure of the group algebra, 
 so computing the number and dimensions of the irreducible representations 
 of a group over such fields is relatively straightforward.
 When examining discrete systems (like crystal structures), integral 
 representations become important. This theory can be thought to be the
integral analog of the preceding one, but additional difficulties
appear soon in its development, as even the most basic tools, like
the capacity to decompose every representation as a direct sum of
irreducible representations, fall apart. Even in the simplest
example, the computation of the indecomposable $\mathbb{Z}[C_2]$-modules,
where $C_2$ is the group with two elements, produces objects
like the two-dimensional permutation module, with the non-trivial element 
of $C_2$ transposing the elements of a basis, that are at the same time
reducible and indecomposable. For this reason the study of integral 
representations has been considered for quite a long time part of the realm of 
number theory, see \cite{cli92}, \cite{HR62a}, \cite{Hoff16} or
\cite{Mill14} as examples. The results presented here add to 
this connection.

We follow the methodology of our previous work, see \cite{AAS}. 
Rather than building the module theory for the ring $\mathbb{Z}[G]$,
as we would for $K[G]$ when $K$ is a field, we rely on the 
existing classification of representations over a number field $K$, 
and from there we study the set of  integral 
representations conjugate to a given $K$-representation $\phi$.
They are called the integral forms of $\phi$. These integral forms
can be described in terms of a combinatorial object, the Bruhat-Tits
building, or BTT, which is a major tool for the study of maximal orders in 
matrix algebras. Note that an integral form is essentially a representation 
whose image falls in one particular maximal order.
The set of integral forms is closely linked to a subcomplex
called the branch of the representation. General tools for computing
branches have been the focus of some recent research, like in 
\cite{A13}, \cite{A16}, \cite{AS16} or \cite{AB19},
and it is our hope that they can shed some 
light on integral representation theory. In fact, this technique is
sufficient, in our opinion, to compute the integral forms for all
two dimensional irreducible representations, as we outlined in
\cite{AAS}, since the list of those groups is rather limited. This requires
writing the images of the generators as explicit matrices, so there is some
work to be done, especially for representations whose Schur index is $2$,
like the quaternion group $Q_8$. See \cite{AB26} for details on the latter. 
We also expect that the geometry of higher dimensional complexes proves helpful
in the general case, although a general method, as we presented in
the two dimensional case, still seems elusive to us.
However, cyclotomic field 
arithmetic plays a major role for the groups considered in this work, 
as is apparent in subsequent sections. This allows us to do
considerably more in their case.

 In all of this work we let 
$H$ be a group with a presentation
\begin{equation}\label{fmg}
H=H(N,M,k)=\langle a,b|a^M=b^N=e, bab^{-1}=a^k\rangle.
\end{equation}
We let $r$ be the order of the class of $k$ as an element of
the multiplicative group $(\mathbb{Z}/M\mathbb{Z})^*$. Then $r$
divides $N$, as $a=b^Nab^{-N}=a^{k^N}$.
This is a particular example of a metacyclic group, i.e.,
an extension of a cyclic group by a cyclic group. 
Equivalently, a metacyclic group is a 
group $G$ that has a cyclic normal subgroup $N$, such that the 
quotient $G/N$ is also cyclic. Not every metacyclic 
group has the form $H(N,M,k)$ as above, for instance the 
quaternion group $Q_8$ does not. However, this family includes 
the dihedral groups, and it is wide enough to include all groups 
for which our main result applies, so we keep this notation throughout.

Let $G$ be an arbitrary finite group, and let $K$
be a number field. Let
$\phi:G\rightarrow\mathrm{GL}_n(K)$ be an absolutely
irreducible $n$-dimensional representation of $G$
defined over $K$. A field of definition of $\phi$ is a minimal
field $K_0$ satisfying $\phi'(G)\subseteq\mathrm{GL}_n(K_0)$ for some
conjugate $\phi'$ of $\phi$. Any field of definition of $\phi$ contains
the set of images $\chi_\phi(G)$ of the character $\chi_\phi$ defined by
$\chi_\phi(g)=\mathrm{tr}\big(\phi(g)\big)$. Therefore, when 
$\mathbb{Q}\big(\chi_\phi(G)\big)$ is a field of definition, it is unique
and we call it the field of definition of $\phi$. In this case we say that
the Schur index of $\phi$ is $1$. More generally, the Schur index of a 
representation $\phi$ is $s$ when $\phi$ can be seen as a representation with 
values in a matrix algebra over an $s^2$-dimensional central division algebra,
see for example \cite[\S2.4]{GS17}.
By an integral form of $\phi$
we mean a $\mathrm{GL}_n(\mathcal{O}_K)$-conjugacy
class of representations that are $\mathrm{GL}_n(K)$-conjugates of $\phi$. 
Integral forms of a representation may fail to exist, see \cite{Se08}.
Finally, we let $\mathcal{G}_K$ denote the class group of $K$,
$\mathcal{G}_K(n)$ its maximal exponent $n$ subgroup, while
$h_K$ and $h_K(n)$ denote their respective orders.
The main result of this work
is the following:

\begin{theo}\label{t0}
    Let $G$ be a group containing a subgroup $H$ that
    is isomorphic to some group $H(N,M,k)$, as in Eq. (\ref{fmg}),
    and assume that the multiplicative order $r$ of $k$ modulo 
    $M$ equals $N$. Let $\phi:G\rightarrow\mathrm{GL}_n(K)$ be an
    absolutely irreducible representation of $G$
    whose restriction $\phi'$ to $H$ is 
    also absolutely
    irreducible. Assume $K$ is the field of
    definition of $\phi'$. Then the reduction modulo a fixed maximal 
    $\mathcal{O}_K$-ideal $\wp$ of every integral form of $\phi$ 
    has the same number $c_\wp$ of non-trivial invariant subspaces.
    Furthermore, the number of integral forms is 
    $h_K(n)c$, where $c\leq \prod_\wp(1+c_\wp)$.
    If $h_K(n)=1$, i.e., if $\mathrm{gcd}(n,h_K)=1$,
    then $c=\prod_\wp(1+c_\wp)$.
\end{theo}

The product above is finite since $c_\wp>0$ for only a finite set of $T$ of
maximal ideals $\wp$. In fact, $c_\wp=0$ as soon as the $\mathcal{O}_\wp$-order 
spanned by the image $\phi(G)$ is maximal. 

Recall that the permutation representation for the symmetric group $S_n$ is
an $n$-dimensional representation defined in terms of the canonical basis by 
$\psi(\sigma)(\mathbf{e}_n)=\mathbf{e}_{\sigma(n)}$. It is easy to show that
this representation is the direct sum of two irreducible representations. The 
trivial representation and an $(n-1)$-dimensional irreducible representation
called the irreducible permutation representation. The preceding theorem implies
that the irreducible permutation representation for $S_p$, for a prime $p$, has
precisely two integral forms, as we show in Ex. \ref{ex55}.
This fact has been known for some time. More 
generally, it is known that the number of
integral forms of this representation for an arbitrary
symmetric group $S_n$ equals the
number of divisors of $n$. 
This was proved independently by Wilhelm
Plesken in \cite{Ple74}, Maurice
Craig in \cite{Cra76}, and reproved via
Bruhat-Tits building techniques 
by Walter Feit in \cite{Fei98}.
For the prime case, the unique invariant subspace
of the reduction modulo $p$ has either dimension $1$
or codimension $1$, depending on the integral form.
Theorem \ref{t0} also applies to the $3$-dimensional
representation of $\mathrm{PSL}(2,7)$. See Ex. \ref{ex56}.
It also applies to the group $H$ itself.

There are two facts that are behind Theorem \ref{t0}. One is a connection
between the integral representations of the group $H$
with the arithmetic of cyclotomic extensions. To be precise, for any extension
$L/K$ of number fields, we let $\mathcal{I}_L$ denote the ideal group of $L$,
and we identlfy the ideal group $\mathcal{I}_K$ of $K$ with a subset of 
$\mathcal{I}_L$ by identifying a fractional ideal $J$ with $\mathcal{O}_LJ$.
We let $\mathcal{P}_K$ denote the subgroup of principal fractional ideals in $K$,
while $\mathcal{Q}_{L/K}$ is the group of fractional ideals in $L$ whose $n$-th
powers belong to $\mathcal{P}_K$. In these terms, we prove the following result:

\begin{theo}\label{t2}
    Let $H=H(N,M,k)$ be a group presented as in Eq. (\ref{fmg}), 
    with $r=N$ as above. Let 
    $\phi:H\rightarrow\mathrm{GL}_n(K)$ be a
    faithful, absolutely irreducible representation of $H$,
    whose field of definition is $K$. Let $L=K(\eta)$, where
    $\eta$ is a primitive $M$-th root of unity. In this case, the
    constant $c_\wp$ defined in Theorem \ref{t0}, for $G=H$,
    is given the formula $c_\wp=e_\wp(L/K)-1$, where $e_\wp(L/K)$ denotes
    the ramification degree. Furthermore, in this case the number
    of integral forms equals the order of the quotient
    $\mathcal{Q}_{L/K}/\mathcal{P}_K$. In particular, the set of
    integral forms is not empty.
\end{theo}

The second fact relates to the structure of the Bruhat-Tits tree and its
interpretation in terms of maximal ideals. Loosely speaking, when
considering an integral representation as a map from $G$ to the group
of units of a maximal order in a matrix algebra, neighboring orders
containing also the image of $G$ can be ``seen'' as non-trivial
invariant subspaces. This implies that, if all representations correspond
to maximal order that are neighbors of each other, then all integral forms
can be seen from a single representation. In this context, our main result has the following form:

\begin{theo}\label{t1}
    Let $H=H(N,M,k)$ be a group presented as in Eq. (\ref{fmg}), 
    with $r=N$ as above. Let 
    $\phi:H\rightarrow\mathrm{GL}_n(K)$ be a
    faithful, absolutely irreducible representation of $H$,
    whose field of definition is $K$. Then, at every place $\wp$ of $K$, 
    the vertices of the local Bruhat Tits building corresponding to maximal 
    orders containing $\phi(H)$ lie in a simplex.
    This simplex is non-trivial precisely at those places $\wp$
    that are ramified for the extension $L/K$, where $L=K(\eta)$ for
    a primitive $M$-th root of unity $\eta$. In this case,
    the number of vertices of the 
    corresponding simplex is the ramification index
    $e_\wp(L/K)$.
\end{theo}

\section{Characterizing the complex representations}

The purpose of this section is to give an explicit list
of characters for all representations of the group
$H(N,M,k)$, presented as in Eq. (\ref{fmg}). 

Let $\nu$ be a root of unity
of order $N/r$. Let $\Xi$ be the set of primitive 
$M$-th roots of unity, and let $\sigma$ be the Galois
automorphism taking each element of $\Xi$ to its $k$-th
power. Finally, let $\mathbf{O}$ be an orbit in $\Xi$
under the action of the group $\langle\sigma\rangle$.
Then we define the representation
$\phi=\phi_{\mathbf{O},\nu}$ as follows:
\begin{equation}\label{rgen}
\phi(a)=\left(\begin{array}{ccccc}
 \xi & 0 & 0 & \cdots & 0 \\
 0 & \xi^k & 0 & \cdots & 0 \\
 0 & 0 & \xi^{k^2} & \cdots & 0 \\
  \vdots & \vdots & \vdots & \ddots & \vdots \\
  0 & 0 & 0 &\cdots &\xi^{k^{r-1}}
\end{array}\right),
\qquad
\phi(b)=\left(\begin{array}{ccccc}
 0 & 1 & 0 & \cdots & 0 \\
 0 & 0 & 1 & \cdots & 0 \\
 0 & 0 & 0 & \cdots & 0 \\
  \vdots & \vdots & \vdots & \ddots & \vdots \\
  0 & 0 & 0 &\cdots & 1 \\
\nu & 0 & 0 &\cdots & 0
\end{array}\right),
\end{equation}
where $\xi\in\mathbf{O}$. Note that $\phi(b^r)=\nu\mathtt{1}$, 
where $\mathtt{1}$ is the identity matrix.

Let $K$ be a field, and let $L$ be an $r$-dimensional $K$-algebra. 
Since $L$ is a $K$-vector space of dimension $r$, we can identify 
the endomorphism ring $\mathbb{M}:=\mathrm{End}_K(L)$ with the matrix 
algebra $\mathbb{M}_r(K)$. We do this in all that follows.

\begin{lem}\label{l38}
    Let $L$ be an $r$-dimensional \'etale commutative $K$-algebra
    over a field $K$, and identify any element $\lambda\in L$ with 
    the $K$-linear map $m_\lambda\in \mathbb{M}$
    defined by $m_\lambda(\mu)=\lambda\mu$. Let $\sigma:L\rightarrow L$
    be a $K$-algebra automorphism of order $r$, whose invariant subalgebra 
    is the $K$ span $K1_L$ of the identity element $1_L\in L$. 
    Regard $\sigma$ as an element in the corresponding endomorphism 
    ring $\mathbb{M}$, as defined above. Then 
    $\mathbb{M}=K[L,\sigma]=\bigoplus_{i=0}^{r-1}L\sigma^i$.
\end{lem}

\begin{proof}
    Write $L=K[\lambda]$. Since $\mathbb{M}$ has dimension $r^2$ as 
    a $K$-vector space, it suffices to prove that 
    $\sum_{i=0}^{r-1}\mu_i\sigma^i=0$, with $\mu_i\in L$, implies
    $\mu_0=\cdots=\mu_{r-1}=0$. Assume 
    $\sum_{i=0}^{r-1}\mu_i\sigma^i=0$. This is equivalent to
    $\sum_{i=0}^{r-1}\mu_i\sigma^i(\mu)=0$ for every $\mu\in L$.
    Extending scalars if needed, we can assume that $K$ is 
    algebraically closed, and therefore $L$ is isomorphic to a cartesian
    product $K^n$ with coordinate-wise operations. Note that the 
    hypothesis that $L$ is \'etale is used here, since it is preserved
    under field extensions. Let $\rho_0,\dots,\rho_{r-1}$ be the minimal 
    idempotents in $L$. Then $\sigma$ acts on the set 
    $\{\rho_0,\dots,\rho_{r-1}\}$. The condition on the invariant subalgebra of 
    $\sigma$ implies that this action is transitive. Renaming, we can assume that
    $\sigma^i(\rho_t)=\rho_{t+i}$, where the index is considered as an integer 
    modulo $r$. It follows that $\sum_{i=0}^{r-1}\mu_i\rho_{i+t}=
    \sum_{i=0}^{r-1}\mu_i\sigma^i(\rho_t)=0$. Postmultiplying by $\rho_{j+t}$ we 
    get $\mu_j\rho_{j+t}=0$. This means that the $(j+t)$-th coordinate of $\mu_j$
    is trivial. Since $t$ was arbitrarily chosen, we conclude that  $\mu_j=0$ as
    claimed.  
\end{proof}

\begin{cor}\label{c211}
    Let $L/K$ be a cyclic extension of degree $r$, 
    and identify any $\lambda\in L$
    with the $K$-linear map $m_\lambda\in \mathbb{M}$
    defined by $m_\lambda(\mu)=\lambda\mu$. Let $\sigma$
    be a generator of the Galois group 
    $\mathrm{Gal}(L/K)\subseteq \mathbb{M}^*$. Then
    $\mathbb{M}=K[L,\sigma]=\bigoplus_{i=0}^{n-1}L\sigma^i$.\qed
\end{cor}

\begin{lem}
    The representations $\phi_{\mathbf{O},\nu}$ defined above
    are absolutely irreducible. Furthermore, if $\phi_{\mathbf{O},\nu}$ 
    and $\phi_{\mathbf{O}',\nu'}$ are conjugates, then
    $\mathbf{O}=\mathbf{O}'$ and $\nu=\nu'$.
\end{lem}

\begin{proof}
  Set $L=K^n$, as before.
  Identify $L$ with the ring of diagonal matrices and set
  $\sigma=\phi_{\mathbf{O},1}(b)$, as in Eq. (\ref{rgen}), i.e., $\sigma$
  is the matrix obtained by replacing $\nu$ by $1$ in $\phi_{\mathbf{O},\nu}(b)$.
  It is easy to see that $\sigma$ thus defined satisfies the hypotheses of Lemma
  \ref{l38}. Since $\phi_{\mathbf{O},\nu}(b)$ is the product of $\sigma$
  and an invertible element of $L$, the first statement is immediate. Now 
  assume $\phi_{\mathbf{O},\nu}$ and $\phi_{\mathbf{O}',\nu'}$ are conjugates. 
  Then the matrices $\phi_{\mathbf{O},\nu}(a)$ and $\phi_{\mathbf{O}',\nu'}(a)$
  have the same eigenvalues. It follows that $\mathbf{O}=\mathbf{O}'$.
  Since the central elements $\phi_{\mathbf{O},\nu}(b^r)=\nu\mathtt{1}$ and
  $\phi_{\mathbf{O}',\nu'}(b^r)=\nu'\mathtt{1}$ are conjugates, then $\nu=\nu'$.
\end{proof}

\begin{lem}\label{l23}
    Every faithful complex representation of $H(N,M,k)$ is a conjugate
    of some representation of the form $\phi_{\mathbf{O},\nu}$.
\end{lem}

\begin{proof}
    Let $\Xi'$ be the set of all $M$-th roots of unity, primitive or otherwise.
    Consider an orbit $\mathbf{O}'\in\Xi'$ under the
    action $\xi\mapsto\xi^k$. All roots in such an orbit have the
    same multiplicative order $M'$, which is a divisor of $M$.
    Let $r'$ be the multiplicative order of $k$ modulo $M'$,
    which is the cardinality of $\mathbf{O}'$,
    and let $\nu$ be an arbitrary $(N/r')$-root of unity. 
    Then $H(N,M,k)$ has a representation $\phi_{\mathbf{O}',\nu'}$
    that factors through $H(N,M',k)$. Then
    the representation $\phi_{\mathbf{O}',\nu'}$  has dimension $r'$, 
    it is absolutely irreducible, and completely determines the pair
    $(\mathbf{O}',\nu')$ by the preceding lemma.
    Since $MN$ is the order of the group $H(N,M,k)$,
    the computation
    $$\sum_{M'|M}(r')^2\cdot\frac{\varphi(M')}{r'}\cdot\frac N{r'}=MN,$$
    where $\varphi$ denotes Euler's function on positive integers,
    proves that every representation of $H(N,M,k)$ has the form 
    $\phi_{\mathbf{O}',\nu'}$. Now the result follows if we observe that
    $\phi_{\mathbf{O}',\nu'}$ is faithful precisely when $M'=M$ and
    $\nu'$ is primitive.
\end{proof}

\begin{lem}\label{l24}
    Let $\xi\in\mathbf{O}$ be an $M$-th root of unity, and
    let $K_0\subseteq L=\mathbb{Q}[\xi]$ be
    the invariant field of the automorphism 
    $\sigma$ given by $\sigma(\xi)=\xi^k$. Then $\phi_{\mathbf{O},\nu}$ 
    is defined over $K_0$ if and only if $\nu\in K_0$ and there
    is an element $\lambda\in L$ satisfying $N_{L/K_0}(\lambda)=\nu$.
    If this condition is satisfied, then $K_0$ is the unique minimal field of
    definition of $\phi_{\mathbf{O},\nu}$. In other words, the values of the
    character generate $K_0$, and the Schur index is $1$.
\end{lem}

\begin{proof}
    By Galois theory, $L=\mathbb{Q}[\xi]$ is an extension of degree $r$ of
    the field $K_0$ and $L=K_0[\xi]$. Assume $\phi_{\mathbf{O},\nu}$
    is defined over $K_0$, i.e., there is a representation 
    $\phi:H(N,M,k)\rightarrow \mathrm{GL}_r(K)$ that is conjugate to
    $\phi_{\mathbf{O},\nu}$ over $L$. 
    Since $\xi$ is an eigenvalue of $\phi(a)$,
    the latter matrix has a minimal polynomial
    equal to the irreducible polynomial
    of $\xi$ over $K_0$. It follows that 
    $\phi(a)$ generates a subalgebra 
    $\mathbb{L}\subseteq\mathbb{M}_r(K)$
    that is isomorphic to $L$. Hence,
    we can identify the vector space $K^r$ with $L$, and therefore
    $\mathbb{M}_r(K)$ with $\mathbb{M}$, in
    a way that $\phi(a)=m_\xi$. Furthermore, under these identifications,
    we have $$\phi(b)m_\xi\phi(b)^{-1}=
    \phi(b)\phi(a)\phi(b)^{-1}=\phi(a)^k=m_{\xi^k}=m_{\sigma(\xi)}.$$
    It follows that $\phi(b)\sigma^{-1}$ commutes
    with $m_{\xi}$. Since $m_{\xi}$ generates a maximal 
    commutative subalgebra of 
    $\mathbb{M}$, then $\phi(b)=m_\lambda\sigma$
    for some $\lambda\in\L$. Then a direct computation shows that 
    $$m_\nu=\nu\mathtt{1}=\phi(b)^r=
    (m_\lambda\sigma)^r=
    m_\lambda m_{\sigma(\lambda)}\cdots
    m_{\sigma^{r-1}(\lambda)}\sigma^r=
    m_{N_{L/K_0}(\lambda)}.$$
    The result follows.
    On the other hand, if the conditions are
    satisfied, we can define a representation 
    $\phi_m:H\rightarrow\mathbb{M}^*$ by $\phi_m(a)=m_\xi$ and 
    $\phi_m(b)=m_\lambda\sigma$, where $K^r$
    is identified with $L$ as before.
    Then, the eigenvalues of $\phi_m(a)$
    are the elements in the Galois orbit
    of $\xi$, while $\phi(b)_m^r=\nu\mathtt{1}$
    by the preceding computation.
    This proves that $\phi_m$ is conjugate to
    $\phi_{\mathbf{O},\nu}$ over $L$.
\end{proof}

\section{On Bruhat-Tits buildings and invariant lattices}

Let $K$ be the field of quotients of a Dedekind domain $\mathcal{O}_K$. 
A full lattice in a finitely-dimensional $K$-vector space $V$
is a finitely generated $\mathcal{O}_K$-module 
$\Lambda\subseteq  V$ that spans $V$. 
An order in a central simple $K$-algebra $\mathfrak{A}$, 
e.g., a matrix algebra, is a unitary subring 
$\mathfrak{D}\subseteq\mathfrak{A}$ that
is also a lattice when we regard $\mathfrak{A}$ as
a vector space. Maximal orders, i.e., orders that are not 
properly contained in any larger order, always exist in
a central simple algebra. Furthermore, any order is contained in 
a maximal order. For instance, the endomorphism ring
$\mathrm{End}_{\mathcal{O}_K}(\Lambda)$, of any lattice
$\Lambda\subseteq K^n$, is a maximal order in the algebra
$\mathbb{M}_n(K)$. Every maximal order in $\mathbb{M}_n(K)$
has this form.
Two lattices $\Lambda$ and $\Lambda'$ satisfy
$\mathrm{End}_{\mathcal{O}_K}(\Lambda)=
\mathrm{End}_{\mathcal{O}_K}(\Lambda')$
if and only if $\Lambda'=J\Lambda$ for some non-zero fractional
$\mathcal{O}_K$-ideal $J\subseteq K$. 
A full lattice $\Lambda\subseteq V$ is completely determined
by the family of completions $\{\Lambda_\nu\}_\nu$,
where $\nu$ runs over the set of non-archimedean places of $K$.
For any pair of lattices  $\Lambda$ and $\Lambda'$, the families
$\{\Lambda_\nu\}_\nu$ and $\{\Lambda'_\nu\}_\nu$ are
coherent, in the sense that $\Lambda_\nu=\Lambda'_\nu$
for all but a finite number of places $\nu$. Finally,
every family $\{\Lambda''(\nu)\}_\nu$ of local lattices that
is coherent with $\{\Lambda'_\nu\}_\nu$, in this sense,
does define a lattice $\Lambda''$.

Now, we let $F$ be an arbitrary local field.
The Bruhat-Tits building is the simplicial complex $B_F$ having
one vertex for every homothety class $[\Lambda]$ of lattices 
$\Lambda\subseteq F^n$, and where the classes 
$[\Lambda_1],\dots,[\Lambda_s]$ form a simplex if and only if, 
up to reorder and rescaling, we have inclusions
$$\pi_K\Lambda_s\subseteq\Lambda_1\subseteq\Lambda_2
\subseteq\cdots\subseteq\Lambda_s.$$

For every basis $\beta=\{\mathbf{v}_1,\dots,\mathbf{v}_n\}$,
we define the apartment $\Sigma_\beta$ as the largest subcomplex
whose vertices correspond to lattices with a basis of the form
$\beta'=\{\pi_K^{h_1}\mathbf{v}_1,\dots,\pi_K^{h_n}\mathbf{v}_n\}$.
Note that these are precisely the lattices that satisfy
$\mathtt{P}_j\Lambda\subseteq\Lambda$ for $j\in\{1,\dots,n\}$,
where $\mathtt{P}_j$ is the idempotent satisfying
$\mathtt{P}_j\mathbf{v}_j=\mathbf{v}_j$ and 
$\mathtt{P}_j\mathbf{v}_i=\mathbf{0}$ for $i\neq j$.
For any order $\mathfrak{H}\subseteq\mathbb{M}_n(K)$,
the set of lattices satisfying $\mathfrak{H}\Lambda
\subseteq\Lambda$ can be described as the vertex set of a 
suitable subcomplex denoted $\Sigma_{\mathfrak{H}}$.
In particular, if $\mathfrak{H}$ is the commutative
order generated by the idempotents defined from a basis 
$\beta$ we have $\Sigma_{\mathfrak{H}}=\Sigma_{\beta}$.
If $\mathfrak{H}$ is the order generated by the image of
a representation $\phi$, then $\Sigma_{\mathfrak{H}}$ plays
a significant role in describing the integral 
representations that are conjugate to $\phi$ over the field.

When $L/K$ is a finite extension, we can associate to every
$K$-basis $\beta$ of $K^n$ a corresponding $L$-basis $\beta_L$
of $L^n$. The apartment $\Sigma_{\beta_L}$ in $B_L$
can be consistently associated to the apartment $\Sigma_\beta$
of $B_K$, and in fact the realizations of both apartments
are naturally homeomorphic as topological spaces. However, when $L/K$
is ramified, the simplices $a$ in $\Sigma_\beta$ do not correspond
to simplices in $B_L$, but to a subcomplex that can be obtained
by subdividing each face of $a$ in a consistent way.
Fig.~\ref{nbh1} in \S\ref{secex} illustrates the case
$n=e(L/K)=3$.

\begin{lem}\label{vint}
    Let $\Lambda$ be a lattice defined over a field
    $K$, let $\mathbf{v}$ is a vector in $\Lambda$ that
    is not in $\pi_K\Lambda$, and let $\mathtt{T}$ be a matrix 
    such that $\mathtt{T}(\mathbf{v})\in\pi_K\Lambda$.
    Let $L/K$ be a finite ramified extension.
    Let $\Lambda'$ be an $\mathcal{O}_L$-lattice 
    corresponding to a vertex defined over $L$
    that falls inside a $K$-simplex that has $[\Lambda]$
    as a vertex. Assume $a\in L$ satisfies  
    $a\mathbf{v}\in\Lambda'$.
    Then $\mathtt{T}(a\mathbf{v})\in\pi_L\Lambda'$.
\end{lem}

\begin{proof}
    Let $\Sigma$ be an apartment containing the simplex. 
    Then there is a basis $\{\mathbf{v}_1,\dots,\mathbf{v}_n\}$ 
    with respect to which the vertices of the simplex 
    correspond to the lattices $\Lambda_s$ with basis 
    $$\{\mathbf{v}_1,\dots,\mathbf{v}_s,\pi_K\mathbf{v}_{s+1},
    \dots,\pi_K\mathbf{v}_n\},$$ for some $s\in\{1,\dots,n\}$.
    We can assume $\Lambda=\Lambda_n$. Then we can assume that
     the inner lattice
     $\Lambda'$ has basis $\{\pi_L^{h_1}\mathbf{v}_1,\dots,
     \pi_L^{h_n}\mathbf{v}_n\}$ with $0=h_1<h_2<\cdots<h_n<e$.
     The reason why the inequalities must be strict is that
     $h_s=h_{s+1}$ puts $\Lambda'$ in the face that is opposite
     to $\Lambda_s$, while $h_n=e$ puts $\Lambda'$ in the face 
     opposing $\Lambda$. Write $\mathbf{v}=
     \sum_{i=1}^nb_i\mathbf{v}_i$ with $b_i\in\mathcal{O}_K$. 
     Let $j$ be the largest index
     for which $b_j\notin\pi_K\mathcal{O}_K$. Since 
     $a\mathbf{v}\in\Lambda'$, then $a\in\pi_L^{h_j}\mathcal{O}_L
     \subseteq\mathcal{O}_L$. Furthermore,
     since $\mathtt{T}(\mathbf{v})\in\pi_K\Lambda$, we can
     write $\mathtt{T}(\mathbf{v})=
     \pi_K\sum_{i=1}^nc_i\mathbf{v}_i$, whence we conclude that
     $\mathtt{T}(a\mathbf{v})=
     \pi_Ka\sum_{i=1}^nc_i\mathbf{v}_i$.
     It suffices to prove that $\pi_L^{h_i+1}$ divides $\pi_Ka$
     for every $i=1,\dots,r$, but this follows from
     the inequality $h_i+1\leq e$.
\end{proof}

\begin{rem}
The conclusion of the preceding Lemma still follows when $[\Lambda']$
is in the border of the simplex, unless it is in the face opposed to
$[\Lambda]$, $j=1$ and $a$ is a unit. However, the
current version will suffice for the examples presented here.
\end{rem}

\begin{lem}\label{l22}
    Assume $K=K_0$ is as in Lemma \ref{l24}.
    Consider the ring $R=
    \mathcal{O}_{K_0}[\mathcal{O}_L,\sigma]\subseteq 
    K_0[L,\sigma]=\mathrm{End}_{K_0}L=\mathbb{M}$.
    Then the $R$-invariant $\mathcal{O}_{K_0}$-lattices
    $\Lambda\subseteq L$ are precisely the Galois invariant
    fractional ideals in $L$.
\end{lem}

\begin{proof}
    An $\mathcal{O}_{K_0}$-lattice in $L$ is 
    $\mathcal{O}_L$-invariant
    if and only if it is a fractional 
    $\mathcal{O}_L$-ideal. If $I$ is a 
    Galois invariant
    fractional $\mathcal{O}_L$-ideal, 
    then it is clearly $R$-invariant.
    On the other hand, an $R$-invariant
     fractional $\mathcal{O}_L$-ideal $I$
     satisfies $\sigma(I)\subseteq I$.
     Applying $\sigma$ to that 
     inclusion, we get a chain
     $I\supseteq \sigma(I)\supseteq
     \sigma^2(I)\supseteq\cdots$.
     The result follows since $\sigma$
     has finite order, whence all inclusions above must be equalities.
\end{proof}

\begin{prop}\label{p34}
    Let $H=H(N,M,k)$, and let $L$, $K_0$, $R$ and $\mathbb{M}\cong\mathbb{M}_r(K_0)$
    be as above. Assume that $\phi_m:H\rightarrow\mathbb{M}^*$ be a representation
    satisfying $\phi_m(a)=m_\xi$ and $\phi_m(b)=m_\lambda\sigma$, where $\xi$, $\lambda$
    and $\sigma$ are as in Lemma \ref{l24}. Then the vertices of 
    the local Bruhat-Tits at $\wp$ corresponding to the homothety classes 
    of $\phi_m(H)$-invariant lattices lie in a simplex with $e=e_\wp(L/K_0)$ vertices.
\end{prop}

\begin{proof}
    Note that 
    $\mathcal{O}_L=\mathcal{O}_{K_0}[\xi]$ for
    a $M$-th root of unity $\xi$.
    According to Lemma \ref{l22}, the invariant lattices are the Galois invariant fractional 
    ideals in $L$. Note that two 
    $\mathcal{O}_L$-fractional ideals in
    $\mathcal{I}_L$
    differing only by a factor in 
    $\mathcal{I}_{K_0}$
    are in the same homothety class.
    Since $L/K$ is Galois, every prime ideal $\wp$ 
    in $\mathcal{O}_K$ generates an ideal $\wp\mathcal{O}_L$ that factors 
    as $\wp\mathcal{O}_L=(\mathfrak{P}_1\cdots\mathfrak{P}_t)^e$
    in $L$, with the factors $\mathfrak{P}_i$
    in the same Galois orbit. 
    We conclude that a Galois invariant fractional 
    ideals is a product of factors of the form $\tilde{\wp}^s$ with
    $\tilde{\wp}=\mathfrak{P}_1\cdots\mathfrak{P}_t$, with one such factor for every
    ramified place of $L/K_0$.
    Furthermore, multiples of $e$ in the exponent
    can be ignored, since they define factors
    in $\mathcal{I}$. Now, it suffices to 
    note that the relation 
    $$\wp\mathcal{O}_L=\tilde{\wp}^e\subseteq\tilde{\wp}^{e-1}\subseteq\cdots
    \subseteq\tilde{\wp}\subseteq\mathcal{O}_L,$$
    proves that the corresponding local lattices are the vertices 
    of a simplex.
\end{proof}

\subparagraph{Proof of Theorem \ref{t1}.}
By Lemma 2.3, $\phi$ is conjugate to some representation $\phi_{\mathbf{O},\nu}$.
Note that the representation 
$\phi$ has dimension $n=r$. 
The condition $N=r$ proves that $\nu=1$, whence
we can take $\lambda=1$, as an element of norm
$\nu$. It follows, by Lemma \ref{l24}, that $K=K_0$ 
is the invariant field of the automorphism sending $\eta$ 
to $\eta^r$. By the proof of Lemma \ref{l24}, $\phi$ is 
conjugate to the representation $\phi_m$ in Prop. \ref{p34}.
This means that $\phi$ and $\phi_m$ have the same character,
whence they are conjugates in $\mathrm{GL}_n(K_0)$.
Then, by Prop. \ref{p34}, every maximal
order containing $\phi(H)$, locally at $\wp$,
is in a simplex with $e=e_\wp(K/L)$ vertices. 
The result follows.\qed

\section{Free lattices and representations}\label{secfl}

A free (full) lattice in an $n$-dimensional
$K$-vector space $V$ is a lattice $\Lambda$
that is isomorphic to $\mathcal{O}_K^n$ as
an $\mathcal{O}_K$-module.

\begin{prop}
    Any ring of the form $\mathcal{O}_K[\alpha]$, where $\alpha$ is integral over $\mathcal{O}_K$,
    is a free lattice.
\end{prop}

\begin{proof}
    If $f(x)=x^n+a_{n-1}x^{n-1}+\dots+a_1x+a_0$
    is the irreducible polynomial of $\alpha$
    over $\mathcal{O}_K$, then the set
    $\{1,\alpha,\dots,\alpha^{n-1}\}$ spans
    $\mathcal{O}_K[\alpha]$, since every power
    $\alpha^n,\alpha^{n+1},\dots$ can be expressed 
    as a combination of lower powers. Furthermore, 
    this set is linearly independent since
    $\alpha$ is not a root of a polynomial of lower
    degree. The result follows.
\end{proof}

\begin{cor}\label{c511}
    If $L$ is a cyclotomic field, and $K$ is any 
    subfield, then $\mathcal{O}_L$ is free as an
    $\mathcal{O}_K$-module. 
\end{cor}

\begin{proof}
    This is a consequence of the fact that the ring
    $\mathcal{O}_L$ equals $\mathbb{Z}[\eta]
    =\mathcal{O}_K[\eta]$
    for any primitive root of unity $\eta$.
    See, for instance, \cite[Prop. 1.2]{W97}.
\end{proof}

The reader should be warned, however, that this is 
not the case for an arbitrary extension of number fields.
For instance, if $K=\mathbb{Q}(\sqrt{-15})$,
and if $L=K(\omega)$, where $\omega$ is
a primitive cubic root of unity, then no
$\mathcal{O}_L$-fractional ideal is free
as an $\mathcal{O}_K$-module, since it is proven in
\cite[Ex. 9.9]{AAS} that the image of $\mathcal{O}_L$ under
the representation $\psi:L\rightarrow\mathbb{M}_2(K)$ defined 
by identifying $L$ with $K^2$ does not have any invariant free 
lattice.

The space $\bigwedge^nV$ is the space generated
by the symbols $\mathbf{v}_1\wedge\dots\wedge 
\mathbf{v}_n$, where the
wedge product $\wedge$ is linear on each variable 
and satisfies $\mathbf{v}\wedge\mathbf{w}=
-\mathbf{w}\wedge \mathbf{v}$. If $n$
equals the dimension of $V$, $\bigwedge^nV$
is a one-dimensional space. If $\Lambda$ is a full 
lattice in $V$, then $\bigwedge^n\Lambda$ is the
lattice in $\bigwedge^nV$ generated by the products
of the form 
$\mathbf{v}_1\wedge\dots\wedge\mathbf{v}_n$ with 
$\mathbf{v}_i\in\Lambda$. For any basis 
$\{\mathbf{e}_1,\dots,\mathbf{e}_1\}$, there
is a unique ideal $J$ satisfying
$\bigwedge^n\Lambda=
J(\mathbf{e}_1\wedge\dots\wedge\mathbf{e}_n)$,
which we call the determinant of $\Lambda$ with 
respect to that basis. If $J'$ is the determinant
with respect to a different basis, then
$J'=aJ$, where $a$ is the determinant of one basis
with respect to the other. In particular, the class of the
determinant is independent of the basis, and we call it the determinant
class of the lattice.

\begin{prop}\label{p52}
    A full lattice $\Lambda$ is free if and only if 
    its determinant class is the class of principal ideals.
\end{prop}

\begin{proof}
    Necessity is immediate since a free lattice is 
    generated by a basis. For the sufficiency, we fix
    a free lattice $\Lambda_0$, say the one generated by the basis $\{\mathbf{e}_1,\dots,\mathbf{e}_1\}$.
    For any other lattice $\Lambda$, we choose, for
    every finite place $\wp$, a local matrix 
    $\mathtt{T}_\wp$ satisfying $\Lambda_\wp=
    \mathtt{T}_\wp\Lambda_{0,\wp}$. Comparing determinants, 
    we prove that the determinant of
    $\mathtt{T}_\wp$ is a unit. Since
    the stabilizer in $\mathrm{GL}_n(K_\wp)$
    of the local lattice $\Lambda_{0,\wp}$ has a
    matrix with determinant $u$ for an arbitrary 
    unit $u$, we can assume that the determinant of
    $\mathtt{T}_\wp$ is one. Now, the result follows 
    from the strong approximation theorem for 
    $\mathrm{SL}_n(K)$.
\end{proof}

\begin{prop}\label{p53}
    If $J$ is a fractional ideal in $L$, then its
    determinant class, as an $\mathcal{O}_K$-lattice
    is $N_{L/K}(J)I_0$, where $I_0$ is the determinant class
    of $\mathcal{O}_L$.
\end{prop}

\begin{proof}
Choose an $L$-id\`ele $a=(a_\mathfrak{P})_\mathfrak{P}$ satisfying 
$J_\mathfrak{P}=a_\mathfrak{P}\mathcal{O}_\mathfrak{P}$ for every place
$\mathfrak{P}$ of $L$. Recall that, for every place $\wp$ of $K$,
the tensor product $L\otimes_{K}K_\wp$ can be identified with the direct
sum $\bigoplus_{\mathfrak{P}|\wp}L_\mathfrak{P}$, over the places of
$L$ lying over $\wp$. Under this identification, the completion at
$\wp$ of the lattice $J$ corresponds to the product
$\bigoplus_{\mathfrak{P}|\wp}a_\mathfrak{P}\mathcal{O}_\mathfrak{P}$,
whose determinant class is $\left(\prod_{\mathfrak{P}|\wp}
N_{L_\mathfrak{P}/K_\wp}(a_\mathfrak{P})\right)I_{0\wp}$.
The result follows.
\end{proof}

\begin{prop}\label{p44}
    Let $\phi:G\rightarrow\mathrm{GL}_n(K)$ be an absolutely irreducible 
    representation of a finite group $G$. Then there is an $h_k(n)$-to-one
    map from the set of integral forms of $\phi$ to the set of
    $\mathcal{O}_K$-maximal orders containing $\phi(G)$ that are 
    $\mathrm{GL}_n(K)$-conjugate to $\mathbb{M}_n(\mathcal{O}_K)$.
\end{prop}

\begin{proof}
    Let $\phi'$ be a conjugate of $\phi$, say 
    $\phi'(g)=\mathbf{T}\phi(g)\mathbf{T}^{-1}$ , where 
    $\mathbf{T}\in\mathrm{GL}_n(K)$. Then $\phi'(G)\subseteq
    \mathrm{GL}_n(\mathcal{O}_K)$, or equivalently, $\phi'(G)\subseteq
    \mathbb{M}_n(\mathcal{O}_K)$, precisely when $\phi(G)\subseteq
    \mathbf{T}^{-1}\mathbb{M}_n(\mathcal{O}_K)\mathbf{T}$. Furthermore,
    $\mathbf{T}^{-1}\mathbb{M}_n(\mathcal{O}_K)\mathbf{T}=
    \mathbf{S}^{-1}\mathbb{M}_n(\mathcal{O}_K)\mathbf{S}$, precisely when
    $\mathbf{N}=\mathbf{S}\mathbf{T}^{-1}$ belongs to the normalizer $N$ 
    of the order $\mathbb{M}_n(\mathcal{O}_K)$. In other words, 
    two integral representations 
    $\phi'(g)=\mathbf{T}\phi(g)\mathbf{T}^{-1}$ and
    $\phi''(g)=\mathbf{S}\phi(g)\mathbf{S}^{-1}$ correspond to the same
    maximal order if and only if they satisfy 
    $\phi''(g)=\mathbf{N}\phi'(g)\mathbf{N}^{-1}$ for an element 
    $\mathbf{N}\in N$. This means that $N$ act on the set of 
    representations defining the same maximal order by conjugation.
    Since the normalizer of any absolutely irreducible representation 
    is the group $K^*$ of scalar matrices, the representations $\phi'$
    and $\phi''$ are $\mathrm{GL}_n(\mathcal{O}_K)$-conjugates if and
    only if $\mathbf{N}\in K^*\mathrm{GL}_n(\mathcal{O}_K)$. 
    It suffices, therefore, to prove that the quotient 
    $N/K^*\mathrm{GL}_n(\mathcal{O}_K)$ has $h_K(n)$ elements.
    
    Note that $\mathbf{N}\in N$ if and only if
    $\mathbf{N} \mathcal{O}_K^n=J\mathcal{O}_K^n$ for some fractional ideal
    $J=J(\mathbf{N})\subseteq K$. This defines a map from $N$ to the ideal 
    group of $K$ whose kernel is the stabilizer of the lattice 
    $\mathcal{O}_K^n$, i.e., the group $\mathrm{GL}_n(\mathcal{O}_K)$.
    We claim that the image $\mathcal{J}$ coincides with the set $\mathcal{J}'$
    of the ideals $J$ whose $n$-th power $J^n$ is principal. This implies that
    $N/K^*\mathrm{GL}_n(\mathcal{O}_K)$ is isomorphic to 
    $\mathcal{G}_K(n)$, and the result follows.

    To prove the claim, we take the determinant of both sides of the identity 
    $\mathbf{N} \mathcal{O}_K^n=J\mathcal{O}_K^n$. We 
    obtain that $\det(\mathbf{N})\in K^*$ generates the ideal $J^n$.
    This proves $\mathcal{J}\subseteq\mathcal{J}'$. To prove 
    $\mathcal{J}'\subseteq\mathcal{J}$, we observe that, if $J^n$ is principal, 
    the lattice $J\mathcal{O}_K^n$ is free since its determinant class is 
    principal, and therefore there is a matrix satisfying 
    $\mathbf{N} \mathcal{O}_K^n=J\mathcal{O}_K^n$. Furthermore,
    $\mathbf{N}\mathbb{M}_n(\mathcal{O}_K)\mathbf{N}^{-1}$ is the endomorphism
    ring of the latter lattice, and therefore it coincides with
    $\mathbb{M}_n(\mathcal{O}_K)$. We conclude that $\mathbf{N}\in N$.
    The result follows.
\end{proof}

\begin{prop}
    Let $\phi:G\rightarrow\mathfrak{D}_\Lambda^*\subseteq\mathrm{GL}_2(K)$ 
    be an absolutely irreducible representation of a finite group $G$,
    where $\mathfrak{D}_\Lambda$ is a maximal order corresponding to a 
    lattice class $[\Lambda]$. Then the number maximal orders containing
    $\phi(G)$ that are endomorphism rings of $\wp$-neighbors of $[\Lambda]$ 
    equals the number of invariant subspaces of $\Lambda/\wp\Lambda$ under 
    the representation $\bar{\phi}:G\rightarrow\bar{\mathfrak{D}}^*$,
    where $\bar{\mathfrak{D}}=\mathfrak{D}_\Lambda/\wp\mathfrak{D}_\Lambda$.
\end{prop}

\begin{proof}
    If $\mathfrak{D}_{\Lambda'}$ is the endomorphism ring of a lattice
    $\Lambda'$ whose homothety class $[\Lambda']$ is a neighbor of 
    $[\Lambda]$, then we can uniquely choose the representatives in a way that
    \begin{equation}\label{hcr}
        \pi_\wp\Lambda\subseteq\Lambda'\subseteq\Lambda.
    \end{equation}
    In this case $\Lambda'/\pi_\wp\Lambda$ is a subspace of 
    $\Lambda/\pi_\wp\Lambda$, that is invariant under $\bar{\phi}(G)$. A 
    lattice $\Lambda'$ satisfying Eq.~(\ref{hcr}) is uniquely determined by 
    the quotient $\Lambda'/\pi_\wp\Lambda$. Conversely, any subspace of 
    $\Lambda/\pi_\wp\Lambda$ is a quotient $\Lambda'/\pi_\wp\Lambda$ for some 
    lattice $\Lambda'$ satisfying Eq. (\ref{hcr}). If $\Lambda'/\pi_\wp\Lambda$ 
    is invariant under $\bar{\phi}(G)$, then $\Lambda'$ is invariant under 
    $\phi(G)$. The result follows.
\end{proof}

\subparagraph{Proof of Theorem \ref{t0}.}
    It follows from Theorem \ref{t1} that at every local place
    $\wp$, the set of vertices of the Bruhat-Tits building corresponding to
    maximal orders containing $\phi(H)$ lie on a simplex. The same holds for
    $\phi(G)$, as the orders containing $\phi(G)$ form a subset of the set
    of orders containing $\phi(H)$. In particular, if we fix one such order
    $\mathfrak{D}$, corresponding to a vertex $v$ at $\wp$, then all local 
    maximal orders containing $\phi(G)$ correspond to vertices that are
    neighbors of $v$. Then, the number of these vertices equals the number
    $c_\wp$ of invariant subspaces of the residual representation 
    $\bar\phi$ as above. This in particular implies that $c_\wp$ is
    independent of the choice of $\mathfrak{D}$, since $1+c_\wp$ is
    the number of vertices in the simplex. The number 
    $\prod_\wp(1+c_\wp)$ is therefore the number of maximal orders 
    containing
    $\phi(G)$. If $c$ is the number of those that are conjugate to
    $\mathbb{M}_n(\mathcal{O}_K)$, the result follows from 
    Proposition \ref{p44}.
\qed

\subparagraph{Proof of Theorem \ref{t2}.}
    The total number of local maximal orders containing $\phi(G)$ equals 
    $c_\wp+1$, at every local place $\wp$. By Th. \ref{t1}, 
    this number 
    equals $e_\wp(L/K)$. This means that the number of 
    global maximal orders
    containing $\phi(G)$ equals $\prod_\wp e_\wp(L/K)=
    |\mathcal{I}_L^{\mathrm{Gal}(L/K)}/\mathcal{I}_K|$, where
    $\mathcal{I}_L^{\mathrm{Gal}(L/K)}$ denotes the group of Galois-invariant
    ideals. Reasoning as in the
    proof of Prop. \ref{p34}, we conclude that the associated lattices
    corresponding to these orders can be identified with the 
    fractional ideals
    in a set of representatives of the quotient 
    $\mathcal{I}_L^{\mathrm{Gal}(L/K)}/\mathcal{I}_K$. Choose a fractional
    ideal $J\in\mathcal{I}_L^{\mathrm{Gal}(L/K)}$. Then its determinant is
    $N_{L/K}(J)=J^n\in\mathcal{I}_K$. Hence, $J$ is a 
    free lattice if and only
    if $J^nI_0$ is a principal ideal, according to Prop. 
    \ref{p52} and Prop.
    \ref{p53}. On the other hand, Cor. \ref{c511} 
    and Prop. \ref{p52} tell us 
    that $I_0$ is principal, as it is the determinant class
    of $\mathcal{O}_L$ as an $\mathcal{O}_K$-module. It follows that $J$ is a 
    free lattice if and only if we have $J^n\in\mathcal{P}_K$, 
    i.e., $J\in\mathcal{Q}_{L/K}$. Note
    that $\mathcal{Q}_{L/K}\subseteq\mathcal{I}_L^{\mathrm{Gal}(L/K)}$, as 
    $n$-th roots of fractional ideals are unique. The maximal order
    corresponding to a fractional ideal 
    $J\in\mathcal{I}_L^{\mathrm{Gal}(L/K)}$, therefore, is
    a conjugate of $\mathbb{M}_n(\mathcal{O}_K)$ precisely 
    when there exists an
    ideal $I\in\mathcal{I}_K$ such that $JI\in\mathcal{Q}_{L/K}$.  Then 
    the number $c$ of relevant maximal orders, as in 
    Th. \ref{t0}, is the order
    of the group $\mathcal{Q}_{L/K}\mathcal{I}_K/\mathcal{I}_K$, 
    or equivalently, the order of 
    $\mathcal{Q}_{L/K}/(\mathcal{Q}_{L/K}\cap\mathcal{I}_K)$.
    Note that the latter intersection is precisely the 
    group $\mathcal{I}_K(n)$
    of ideals whose $n$-th power is principal. Then the result follows from
    Prop. \ref{p44}, since
    $$|\mathcal{Q}_{L/K}/\mathcal{P}_K|=
    |\mathcal{Q}_{L/K}/\mathcal{I}_K(n)|
    |\mathcal{I}_K(n)/\mathcal{P}_K|=h_K(n)
    |\mathcal{Q}_{L/K}/\mathcal{I}_K(n)|.$$
    \qed

\section{Examples}\label{secex}

For the examples in this section, by convention we denote
matrices by uppercase boldface letters, while vectors are
lower case boldface letters. Note that the extension field
$L$ is often identified with the space of vectors, as much as
a subring of the matrix ring, so notations are chosen accordingly.
Places in number fields are denoted as boldface versions of the 
rational prime they contain, with a suitable subindex. 

\begin{ex}\label{e0}
    Let $H=H(4,15,2)$, and let 
    $\phi=\phi_{\mathbf{O},1}$ with
    $\mathbf{O}=\{\eta,\eta^2,\eta^4,\eta^8\}$, where $\eta$ is a
    root of unity of order $15$.
    This representation is defined over the field
    $K=\mathbb{Q}(\eta+\eta^2+\eta^4+\eta^8)=
    \mathbb{Q}(\sqrt{-15})$. We claim that the only place that ramifies 
    over the extension $L/K$, for $L=\mathbb{Q}(\eta)$, is the place
    $\mathbf{5}_K=(5,\sqrt{-15})$ lying over $5$.
    In fact, $L/\mathbb{Q}$ is ramified only at $3$ and $5$, and the local 
    ramification indices can be computed by considering the sub-fields 
    generated by the roots of unity of degree $3$ and $5$. This gives
    $e_3(L/\mathbb{Q})=2$ and $e_5(L/\mathbb{Q})=4$. Comparing those with
    $e_3(K/\mathbb{Q})=2$ and $e_5(K/\mathbb{Q})=2$, the result follows.
    Note that the fact that there is a unique place over $5$ follows since
    this is the case for the sub-field $F=\mathbb{Q}(\sqrt{-3})$, and $L/F$
    is fully ramified at the place over $5$.
    Since $L$ has a unique place over $\mathbf{5}_K$, denoted 
    $\mathbf{5}_L$, and the ramification index equals $2$, there
    are precisely two maximal orders, $\mathfrak{D}$
    and $\mathfrak{D}'$ containing the
    image of $\phi$ in this case, and they are 
    $\mathbf{5}_K$-neighbors. In fact, one corresponds to
    the trivial fractional ideal of $L$ and the other to 
    $J=\mathbf{5}_L$. Since $N_{L/K}(J)=J^4\cap K=\mathbf{5}_K^2$,
    the local distance at $\mathbf{5}_K$ between $\mathfrak{D}$ and 
    $\mathfrak{D}'$ is the square of the Frobenius map at that place, 
    and therefore
    it is trivial. We conclude that both orders, $\mathfrak{D}$ and
    $\mathfrak{D}'$, are conjugates of $\mathbb{M}_4(\mathcal{O}_K)$,
    and therefore each corresponds to $h_K(2)=2$ integral forms.
\end{ex}

\begin{ex}\label{e1}
   Consider the group $H(3,7,2)$, and the representation
   $\phi=\phi_{\mathbf{O},1}$, where 
   $\mathbf{O}=\{\eta,\eta^2,\eta^4\}$, with $\eta=e^{2\pi i/7}$. 
   In this case $K=\mathbb{Q}(\sqrt{-7})$ has class number $1$, 
   and the extension $L/K$ is fully ramified at the place 
   $\mathbf{7}_1=(\sqrt{-7})$, and unramified at 
   all other places. We conclude as before that the representation $\phi$
   has $3$ integral forms.

  Note that the local field $\mathbb{Q}_7$ contains the primitive
  cubic roots of unity, since $\mathbb{Q}[i]/\mathbb{Q}$ splits at $7$.
  Furthermore, for any matrix $\mathtt{U}$
  with minimal polynomial $X^3-1$, the algebra $\mathbb{Q}_7[\mathtt{U}]$
  is isomorphic to the product algebra $\mathbb{Q}_7^3$,
  and the order generated by $\mathtt{U}$ contains the idempotents
  $\frac13(\mathtt{1}+\mathtt{U}+\mathtt{U}^2)$,
  $\frac13(\mathtt{1}+\omega\mathtt{U}+\omega^2\mathtt{U}^2)$
  and $\frac13(\mathtt{1}+\omega^2\mathtt{U}+\omega\mathtt{U}^2)$,
  since $\frac13$ is an integer in $\mathbb{Q}_7$. We conclude that
  all maximal orders containing $\mathtt{U}$ correspond to the vertices in
  the same apartment. In fact, the same result holds over any 
  extension. We 
  conclude that the maximal $\mathcal{O}_{K'}$-orders containing $\phi(H)$
  correspond to the vertices in a convex polygon contained 
  in an apartment, 
  at any place over $7$, for any field extension $K'/K$.
  The local geometry at the corners of the simplex for
  $K_{\mathbf{7}_1}$ shows that the corresponding residual 
  representation always has one invariant one-dimensional subspace 
  and one invariant two-dimensional subspace.
  We conclude that the same holds for any extension.
  In other words, the local branch is always a triangle
  with the same vertices. Of course, if $K'/K$ is ramified at the
  places over $7$, the triangle contains more than one simplex.
  Fig.~\ref{nbh1} illustrates the example in which 
  $K'=\mathbb{Q}(\sqrt[6]{-7})$.
  \begin{figure}[h!]\label{f1}
\begin{center}
\begin{tikzpicture}[scale=1]
\draw (0,0)--(3,0); \filldraw (0,0) circle(1.5pt); 
\filldraw (1,0) circle(1.5pt); \filldraw (2,0) circle(1.5pt); 
\filldraw (3,0) circle(1.5pt);
\draw (1,0) -- ++(60:2);\draw (2,0) -- ++(120:2);
\draw (1,0) -- ++(120:1);\draw (2,0) -- ++(60:1);
\filldraw (1.5,0.877) circle(1.5pt);
\draw (0.5,0.877) -- ++(0:2);
\filldraw (0.5,0.877) circle(1.5pt);\filldraw (2.5,0.877) circle(1.5pt);
\draw (0,0) -- ++(60:3);\draw (3,0) -- ++(120:3);
\draw (1,1.74)-- ++(0:1); \filldraw (1,1.74) circle(1.5pt); 
\filldraw (2,1.74) circle(1.5pt);
\filldraw (1.5,2.63) circle(1.5pt);
\node [above] at (1.5,0.78) {$\phantom{xxx}{}_{{}_{w_0}}$};
\end{tikzpicture}
\end{center}
\caption{The vertices corresponding to maximal orders
containing the representation $H$ in Example \ref{e1}
over the field $\mathbb{Q}(\sqrt[6]{-7})$.}\label{nbh1}
\end{figure}
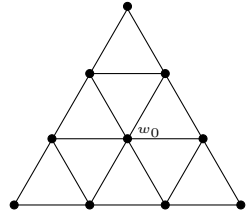
In fact, it can be seen that the number of such maximal orders
is $\prod_{i=1}^N(e_i^2+e_i+1)$, where $e_1,\dots,e_N$ are the
ramification degrees of the places of $K´$ lying over 
$\mathbf{7}_1$.   
\end{ex}

\begin{ex}\label{e2}
   Let $K=\mathbb{Q}(\sqrt{-3})$, and let $L$ be the field generated 
   by the ninth-roots of unity. Then the extension $L/K$ is fully 
   ramified at the unique place $\mathbf{3}_1=(u)$ over $3$, where
   $u=\omega-1$, and it is unramified elsewhere.
   Let $\eta\in L$ be the primitive ninth root of unity,
   so that $w=\eta-1$ is a uniformizing parameter of
   the unique place of $L$ over $\mathbf{3}_1$.
   Consider the group $H=H(3,9,3)$ and the representation
   $\phi=\phi_{\mathbf{O},1}$, with 
   $\mathbf{O}=\{\eta,\eta^4,\eta^7\}$. 
   We denote by $\mathtt{H}$ the matrix corresponding to multiplication 
   by $\eta$, and by $\mathtt{E}$ the matrix corresponding to the 
   generator $\epsilon$ of the Galois group.
   Again, there are exactly $3$ integral representations that are 
   $K$-conjugate to $\phi$ and $K$ contains the cubic roots
   of unity, however, the matrix $\mathtt{E}$ with minimal polynomial
   $X^3-1$ no longer generates the ring of integers of
   the corresponding \'etale algebra. For instance,
   the idempotents  $\frac13(\mathtt{1}+\mathtt{E}+\mathtt{E}^2)$,
  $\frac13(\mathtt{1}+\omega\mathtt{E}+\omega^2\mathtt{E}^2)$
  and $\frac13(\mathtt{1}+\omega^2\mathtt{E}+\omega\mathtt{E}^2)$
  are integral, but they fail to belong to the ring
  $\mathcal{O}_{\mathbf{3}_1}[\mathtt{E}]$.
  The three lattices that are invariant over
  the ring $R=\mathcal{O}_{\mathbf{3}_1}[\mathtt{H},\mathtt{E}]$
  have, respectively, bases $\{\mathbf{o},\mathbf{w},\mathbf{w}^2\}$,
  $\{u\mathbf{o},\mathbf{w},\mathbf{w}^2\}$,
  and $\{u\mathbf{o},u\mathbf{w},\mathbf{w}^2\}$,
  where $\mathbf{o}$, $\mathbf{w}$ and $\mathbf{w}^2$
  denote the vector corresponding to $1$, $w$ and $w^2$.
  By passing to a ramified abelian cubic extension like
  $K'=\mathbb{Q}[\sqrt[6]3]$, the lattices invariant under this ring still 
  correspond to the vertices of the complex in Fig.~\ref{nbh1}. However,
  since the matrix $\mathtt{W}=\mathtt{H}-\mathtt{1}$, which denotes 
  multiplication by $w$, satisfies the hypotheses of Lemma
  \ref{vint} for a different basis vector on the lattice corresponding
  to each vertex, we conclude that $\mathtt{H}$ acts trivially on
  the residual space of the lattice $\Lambda'$ corresponding to the 
  central vertex $w_0$. Since $\mathtt{E}$
  acts as the generator of the Galois group, it acts trivially
  on the residual space of the lattice with basis 
  $\{\mathbf{o},\mathbf{w},\mathbf{w}^2\}$, which corresponds
  to the ring of integers $\mathcal{O}_L$, and therefore also
  on $\Lambda'$ by another application of Lemma \ref{vint}.
  We conclude that every neighbor of the central vertex
  corresponds to a lattice that is invariant under $\phi(H)$.
  In this case there are, therefore, integral forms that correspond to 
  vertices outside the apartment.
\end{ex}

\begin{ex}
Consider the group $H(6,3,-1)$, and the representation
$\phi=\phi_{\mathbf{O},\nu}$, where $\mathbf{O}=\{\omega,\omega^2\}$
and $\nu=\omega=e^{2\pi i/3}$. Then the character $\chi_\phi$ has 
$K=\mathbb{Q}(\sqrt{-3})$ as its field of definition. 
Since $\phi$ has already coefficients in $K$, as given 
in Eq. (\ref{rgen}), we just compute the maximal orders
containing $\phi$ using the method described in \cite{AAS}. Let
$$\mathtt{R}=\phi(a)=\sbmattrix \omega00{\omega^2},\qquad
\mathtt{S}=\phi(b)=\sbmattrix 01\omega0.$$
The eigenvalues of $\mathtt{S}$ are $\pm\omega^2$.
Since $\omega^2-(-\omega^2)=2\omega^2$ is a unit
at the unique place over $3$, $\mathtt{S}$ is contained precisely 
in the maximal orders corresponding to vertices in a path $\mathfrak{p}$
having the roots of $z=\frac1{z\omega}$, i.e., $\pm\omega$, 
as visual limits. On the other hand, the maximal orders
containing $\mathtt{R}$ are those corresponding to vertices
whose distance to the standard apartment $\mathfrak{a}_0$ 
is the valuation of $\omega-\omega^2=\omega(1-\omega)$, 
which is a uniformizer. We conclude that the maximal orders
containing $\mathtt{R}$ are those corresponding to vertices
in $\mathfrak{a}_0$ or their neighbors. Since $\mathfrak{p}$
and $\mathfrak{a}_0$ meet at a single vertex, we conclude that
$\phi(H)$ is contained in precisely $3$ maximal orders,
the order $\mathfrak{D}_0=\mathbb{M}_2(\mathcal{O}_K)$ 
corresponding to the intersection, and the maximal orders
corresponding to the neighbors at either side in the path
$\mathfrak{p}$. This shows that the condition $r=N$ cannot be eliminated
from Th. \ref{t1}.
\end{ex}
 
\begin{ex}\label{ex55}
    Consider the symmetric group $G=S_p$, where $p\geq3$ a prime number, 
    and consider the representation $\phi:S_p\rightarrow GL_p(\mathbb{Q})$
    defined by $\phi(\sigma)(\mathbf{e}_i)=\mathbf{e}_{\sigma(i)}$, on 
    the canonical basis $\{\mathbf{e}_1,\dots,\mathbf{e}_p\}$ of 
    $\mathbb{Q}^p$. The subspace $W$ generated by the vector 
    $\mathbf{u}=\mathbf{e}_1+\cdots +\mathbf{e}_p$ is invariant, so we can 
    decompose $\phi$ as a direct sum of two representations with one of them, 
    denoted $\phi_1$, defined on a $(p-1)$-dimensional space generated by $p$
    vectors $\bar{\mathbf{e}}_i$ with zero sum.  Now consider a primitive p-th root of
    unity $\eta_p$, and note that we have a surjective linear transformation 
    from $\mathbb{Q}^{p}$ to $\mathbb{Q}(\eta_p)$ given by 
    $\mathbf{e}_i\mapsto \eta_p^{i}$, whose kernel is the subspace $W$. This 
    defines a representation that we can 
    identify with $\phi_1$, and we do so in the sequel.  Consider the 
    bijection $\eta_p^{s}\mapsto\eta_p^{rs}$, where $r$ is a 
    generator of $\mathbb{F}_p^*$. This extends to a linear map on 
    $\mathbb{Q}(\eta_p)$ having the form $\phi_1(\lambda)$, 
    where $\lambda\in S_p$ is a $(p-1)$-cycle fixing only $p$, as 
    $\eta_p^{p}=1$ is Galois-invariant. On the other hand, the cycle 
    $\rho=(1\;2\;3\;\dots\;p)$ 
    corresponds to the bijection 
    $\phi_1(\rho)\big(\eta_p^{s}\big)=\eta_p^{s+1}$, which has order $p$.
    Note that $\phi_1$ is faithful, as so is $\phi$, so the computation
    $$\phi_1\big(\lambda\circ\rho\circ\lambda^{-1}\big)
    \left(\eta_p^{s}\right)
    =\phi_1\big(\lambda\circ\rho\big)\left(\eta_p^{sr^{-1}}\right)
    =\phi_1(\lambda)\left(\eta_p^{sr^{-1}+1}\right)
    =\eta_p^{s+r}=\phi_1\left(\rho^{r}\right)(\eta_p^s)$$ 
    shows that the group $H=\langle \lambda,\rho\rangle$ generated by 
    $\lambda$ and $\rho$ is isomorphic to $H(p-1,p,r)$. Furthermore, the 
    restriction of the representation $\phi_1$ to the group 
    $H$ is irreducible by a straightforward application 
    of Lemma \ref{l38}. In fact,
    $\phi_1(\lambda)$ is a generator of the Galois group and $\phi_1(\rho)$
    is multiplication by a generator of the degree $p-1$ extension
    $\mathbb{Q}(\eta)/\mathbb{Q}$. In particular, Th. \ref{t0} applies
    and we can compute the number of integral representations from the 
    number of invariant subspaces of the residual representation 
    at the only
    place that ramifies at $\mathbb{Q}(\eta_p)/\mathbb{Q}$, namely $p$. 

    Let us denote the residual representation by $\bar{\phi}_1$.
    Since $p$ is totally ramified for $\mathbb{Q}(\eta_p)/\mathbb{Q}$,
    the $\bar{\phi}_1(H)$ invariant subspaces lie in a maximal flag.
    More precisely, if $\wp=(\eta_p-1)\mathbb{Z}[\eta_p]$ denotes 
    the unique prime ideal over $p$, then the invariant subspaces
    are the reductions modulo $p$ of the powers of $\wp$, namely
    $\bar{\wp},\bar{\wp}^{2},\dots,\bar{\wp}^{p-2}$. 
    Note that $\bar{\wp}^{p-1}=\{0\}$.
    
    Now, in order to compute the set of maximal orders containing 
    $\phi_1(S_p)$, we have to study each of the subspaces 
    $\bar{\wp}^{l}$ and 
    check if they are invariant under the whole action of $S_p$. 
    Note that the
    transposition $\tau=(1\;p)$ satisfies 
    $S_p=\langle\rho,\tau\rangle$. We 
    already know that $\bar{\phi}_1(\rho)$ leaves invariant all spaces 
    $\bar{\wp}^{l}$, so we only need to study $\bar{\phi}_1(\tau)$. 
    Note that $\tau$ acts on $\mathbb{Q}(\eta_p)$ as follows: 
    $$\phi_1(\tau)(\alpha_1\eta_p+\alpha_2\eta_p^2+\dots+
    \alpha_{p-1}\eta_p^{p-1}+\alpha_p\eta_p^p)=
    \alpha_p\eta_p+\alpha_2\eta_p^2+
    \dots+\alpha_{p-1}\eta_p^{p-1}+\alpha_1\eta_p^p.$$
    Let $\epsilon$ denote the reduction modulo $p$ of $\eta_p-1$. 
    Then $\{\epsilon^l,\dots,\epsilon^{p-2}\}$ is a basis
    of $\bar{\wp}^l$. 
    Note that for $2\leq l \leq p-1$ we have
    $$\phi_1(\tau)\left((\eta_p-1)^l\right)=\phi_1(\tau)
    \left(\sum_{k=0}^{l}\binom{l}{k}\eta_p^{k}(-1)^{l-k}\right)=
    (-1)^l\eta_p+l(-1)^{l-1}+
    \sum_{k=2}^{l}\binom{l}{k}\eta_p^{k}(-1)^{l-k}$$
    $$=(\eta_p-1)^{l}+(-1)^{l}(\eta_p-1)+l(-1)^{l-1}(1-\eta_p).$$
    We conclude that $\bar{\phi}_1(\tau)\left(\epsilon^l\right)=\epsilon^l
    +(-1)^l(l+1)\epsilon$. In particular, the element 
    $\bar{\phi}_1(\tau)\left(\epsilon^{l}\right)
    =\epsilon^l+(-1)^l(l+1)\epsilon$ is a local generator
    of the maximal ideal
    $\bar{\wp}$, and therefore it does not belong to 
    $\bar{\wp}^l$. The latter
    is not, therefore, an invariant subspace. On the other hand, 
    the space $\bar{\wp}$ is invariant since 
    $\bar{\phi}_1(\tau)(\epsilon)=-\epsilon$,
    and the image of $\epsilon^{l}$ is also in this ideal 
    as previously seen.
    We conclude that there are exactly two integral representations. 
\end{ex}

\begin{ex}\label{ex56}
    The group $H=H(3,7,2)$ in Ex. \ref{e0} embeds into the simple group 
    $PSL_2(7)$. Klein's three-dimensional representation $\phi$
    of $PSL_2(7)$ is faithful, and so must be, therefore, its restriction 
    $\phi_H$ to $H$. We conclude that this restriction is irreducible. The
    restriction of $\phi_H$ is defined over the quadratic  field
    $K=\mathbb{Q}(\sqrt{-7})$. We claim that $\phi$ is also defined 
    over $K$. Since $PSL_2(7)$ has only two three-dimensional representations, 
    it suffices to see that the Schur index is $1$. If not, there should be a 
    faithful representation of $PSL_2(7)$ with images on a $9$-dimensional 
    division $K$-algebra, but then the same would hold for $H$, and this is not
    the case since the Schur index for the latter is $1$. This proves the claim.
    Since $K$ has class number $1$, then $\phi$ has an integral form. We conclude 
    that Theorem \ref{t0} applies to $\phi$.
\end{ex}

'

\end{document}